\documentclass[11pt]{amsart}

\usepackage{amsfonts,amsmath,graphicx}

\let\l=\lambda
\newcommand{\bbZ}{{\ensuremath{\mathbb Z}} }
\def\ceil#1{\lceil #1 \rceil}
\newcommand{\bbT}{{\ensuremath{\mathbb T}} }
\def\Tree{{\bbT}}
\newcommand{\bbN}{{\ensuremath{\mathbb N}} }
\newcommand{\cF}{\ensuremath{\mathcal F}}
\let\O=\Omega
\newcommand{\bbR}{{\ensuremath{\mathbb R}} }
\newcommand{\bbP}{{\ensuremath{\mathbb P}} }
\let\t=\tau
\let\s=\sigma
\let\h=\eta
\let\a=\alpha
\let\b=\beta
\let\g=\gamma
\def\thsp{\thinspace}
\def\tc{\thsp | \thsp}
\def\ie{\hbox{\it i.e.,\,}}
\def\bigno{\bigskip\noindent}
\let\d=\delta
\newcommand{\eps}{\epsilon}
\def\medno{\medskip\noindent}
\let\neper=e
\def\nep#1{ \neper^{#1}}
\newcommand{\lm}{\lambda_-}
\newcommand{\lp}{\lambda_+}
\def\inte#1{\lfloor #1 \rfloor}
\let\G=\Gamma

\newtheorem{theorem}{Theorem}[section]
\newtheorem{lemma}[theorem]{Lemma}
\newtheorem{proposition}[theorem]{Proposition}
\newtheorem{corollary}[theorem]{Corollary}
\newtheorem{definition}[theorem]{Definition}

\begin{document}

\title{The multi-state hard core model on a regular tree}

\date{\today}

\author[]{David Galvin} \address{David Galvin, Department of Mathematics, University of Notre Dame, South Bend IN 46556}
\email{dgalvin1@nd.edu}\thanks{}
 \author[]{Fabio Martinelli}
\address{Fabio Martinelli, Dip. Matematica, Universita' di Roma Tre, L.go S. Murialdo 1,
00146 Roma, Italy} \email{martin@mat.uniroma3.it}\thanks{}
 \author[]{Kavita Ramanan} \address{Kavita Ramanan, Division of Applied Mathematics, Brown University, Providence RI 02912}
 \email{Kavita\_Ramanan@brown.edu}\thanks{}
 \author[]{Prasad Tetali} \address{Prasad Tetali, School of Mathematics and School of Computer Science, Georgia  Institute of Technology, Atlanta GA 30332}
 \email{tetali@math.gatech.edu}\thanks{
 This work is supported in part by NSF Grants
  DMS-0401239 and DMS-0701043 (PT), CMMI-0928154 and DMI-0728064 (KR), and DMS-0111298 (DG), by the NSA (DG), and by the European Research Council ``Advanced Grant''  PTRELSS 228032 (FM)}

\begin{abstract}
The classical hard core model from statistical physics, with activity $\lambda > 0$ and  capacity $C=1$, on a graph $G$, concerns a probability  measure on the set ${\mathcal I}(G)$ of independent sets  of $G$, with the measure of each independent set $I\in {\mathcal I}(G)$ being proportional to $\lambda^{|I|}$.
Ramanan et al.\ proposed a generalization of the hard core model as an idealized model of multicasting in communication networks. In this generalization, the {\em multi-state} hard core model, the capacity $C$ is allowed to be a positive integer, and a configuration in the model is an assignment of states from $\{0,\ldots,C\}$ to $V(G)$ (the set of nodes of $G$) subject to the constraint that the states of adjacent nodes may not sum to more than $C$. The activity associated to state $i$ is $\lambda^{i}$, so that the probability of a configuration $\sigma:V(G)\rightarrow \{0,\ldots, C\}$ is proportional to $\lambda^{\sum_{v \in V(G)} \sigma(v)}$.

In this work, we consider this generalization when $G$ is an infinite rooted $b$-ary tree and prove rigorously some of the conjectures made by Ramanan et al. In particular, we show that the $C=2$ model exhibits a (first-order) phase transition at a larger value of $\lambda$ than the $C=1$ model exhibits its (second-order) phase transition.
In addition, for large $b$ we identify a short interval of values for $\l$ above which the model exhibits phase
co-existence and below which there is phase uniqueness. For odd $C$,  this transition occurs in the region of $\l= (e/b)^{1/\ceil{C/2}}$,
while for  even $C$, it occurs around $\l=(\log b/b(C+2))^{2/(C+2)}$. In the latter case, the transition is first-order.
\end{abstract}

\maketitle

\noindent
{\bf Keywords}.
Gibbs measures, hard core model, multicasting, phase transition, loss networks

\medskip

\noindent
{\bf AMS Classification}.
Primary: 82B20, 82B26, 60K35; secondary: 90B15, 05C99.

\pagestyle{myheadings}
\thispagestyle{plain}
\markboth{D. GALVIN, F. MARTINELLI, K. RAMANAN AND P. TETALI}{THE MULTI-STATE HARD CORE MODEL ON A REGULAR TREE}

\section{Introduction}

\subsection{The Multi-State Hard Core Model}

Let $G = (V,E)$ be a finite or countably infinite  graph without loops,
and let $S$ be a finite set. We refer to the elements of $S$
as {\em states}.  Many stochastic processes on $S^{V}$ that
arise in applications are subject to ``hard constraints'' that
prohibit certain values of $S$ from being adjacent to one another
in the graph $G$.  Such processes only attain configurations that
lie in a certain feasible subset of $S^{V}$.
 A generic example is the {\em hard core model},
which has state space $S = \{0,1\}$ and imposes the constraint
that no two adjacent vertices in the graph
 can both have the state $1$.
In other words, the set of feasible configurations for the
hard core model on a graph $G$ is
$\{\sigma \in \{0,1\}^V: \sigma_x + \sigma_y \leq 1~\mbox{for every}~xy \in E \}$, or, equivalently,
the collection of independent sets of the graph $G$.
Processes with such \emph{hard constraints} arise in fields as
diverse as combinatorics, statistical mechanics and
telecommunications.
In particular, the hard core
model arises in the study of random independent sets of a graph
\cite{briwin02,galkah}, the study of gas molecules on a lattice
\cite{bax}, and in the analysis of multicasting in telecommunication networks
\cite{kel,lou,ramzie02}.

In this work, we consider a generalization of the hard core model,
which we refer to as the {\em multi-state} hard core model,
in which the state space is
$$
S_C = \{0, 1, 2, \ldots, C\}\,,
$$
for some integer $C \geq 1$, and
the set of allowable configurations is given by
$$
  \Omega_{G} = \{ \sigma \in S_C^V: \sigma_x + \sigma_y \leq C ~\mbox{for every}~
xy \in E\}.
$$
When $G$ is the $d$-dimensional lattice $\bbZ^d$,
this model was introduced and studied by Mazel and Suhov in \cite{mazsuh},
motivated by applications in statistical physics.
In our work, we  focus on the case where $G$ is
an infinite rooted $b$-ary tree (\ie an infinite graph without cycles in
which each vertex has exactly $b+1$ edges incident to it, except for one
distinguished vertex called the root which has $b$ edges incident to it),
which we denote by $\Tree^b$.

On the tree,
this model was studied by Ramanan et al.\ in  \cite{ramzie02} as an idealized
example of multicasting on a regular tree network, each of whose edges has
 the same capacity $C$.  In communications, multicasting arises when,
instead of having a simple end-to-end connection, a transmission is made
from a single site to a group of individuals  \cite{ballfrancrow93}.
An important performance measure of interest is the probability of
packet loss for a given routing protocol \cite{yanwan98}.
As in \cite{ramzie02}, here we consider an idealized  model in which
the routing is simple in the sense that  nodes
multicast only to their nearest neighbors, and study the impact
of the connectivity of the network (\ie the value of $b$) and
the arrival rate on the blocking (or packet loss) probabilities.
The state $\sigma_v$ of any node or vertex $v \in V$ represents the number
of active multicast calls present at that node.
Multicast calls are assumed to arrive at each node as a Poisson process
with rate $\lambda$ and require one unit of capacity on  each of the
$b+1$ edges emanating from that node.  If this capacity is available,
then the call is accepted and the number of active multicast calls at that
node increases by one, while if the required capacity is not
available, then the state of the node remains unchanged and the call
is said to be blocked or lost.  Calls that are accepted
require a random amount of service and then depart the system.
Service requirements of  calls are assumed to be independent and
identically distributed (without loss of generality with mean $1$), and
independent of the arrival process.
 This model is a special case
of a loss network (see \cite{kel} for a general survey of loss networks
and \cite{ lueramzie06,ramzie02} for connections with this particular model).

For a finite graph $G$ and arrival rate $\lambda$,
it is well-known
that the associated stochastic process has  a unique
stationary distribution $\mu_{G,\lambda}$ on $\Omega_{G}$ that is
given explicitly by
\begin{equation}
\label{formula-statdist}
  \mu_{G, \lambda} \doteq \dfrac{1}{Z_{G, \lambda}} \prod_{v \in V} \lambda_{\sigma_v}
\qquad
\mbox{ for } \sigma \in \Omega_{G},
\end{equation}
where $Z_{G, \lambda}$ is the corresponding normalizing constant (partition function)
$Z_{G, \lambda} \doteq \sum_{\sigma \in \Omega_{G}}
\prod_{v \in V} \lambda_{\sigma_v}$, where the form of $\lambda_{i}$ depends on how the multicast calls are served. If the calls are assumed to be served in a first-come first-served
manner at each node (see \cite{kel}), then we have
\[ \lambda_{i} \doteq \dfrac{\lambda^i}{i !}, \qquad i = 0, \ldots, C,  \]
If they are served using the processor sharing scheduling
discipline at each node (see \cite{kleibook2}), then we have
\begin{equation}
\label{def-lambda2}
\lambda_{i} \doteq \lambda^i, \qquad i = 0, \ldots, C.
\end{equation}
Here we adopt the usual
convention that
$0! = 1$, so that $\lambda_0 = 1$ in both models, and we will
sometimes refer to the arrival rate $\lambda$ as the {\em activity}. In this paper (as in \cite{mazsuh}) our $\lambda_i$'s  will always be as defined in (\ref{def-lambda2}).
Thus, our exclusive focus will be the study of the multi-state hard core
model on a $b$-ary tree $\Tree^b$ with activities given by (\ref{def-lambda2}).

\subsection{Gibbs Measures and Phase Transitions}

Although there is an explicit expression \eqref{formula-statdist}
for the stationary distribution on a finite graph, the computational
complexity of calculating the normalization constant for large graphs
limits the applicability of this formula.  Thus, in order to gain insight
into the behavior of these measures on large graphs,
it is often useful to consider the associated Gibbs measure on an infinite
graph.  Roughly speaking, a Gibbs measure on an
infinite graph $G$ associated with an activity
$\lambda$ is characterized by the property that the
distribution of the configuration on any {\em finite} subset $U$ of $V$,
conditioned on the complement, is equal to the
regular conditional probability of the measure
$\mu_{G[\bar{U}], \lambda}$ on the restriction $G[\bar{U}]$ of the graph
$G$ to the closure $\bar{U} = U \cup \partial U$ of $U$,
given the configuration on the
boundary $\partial U$  of $U$ (see Definition \ref{def-gibbs} below
for a more precise formulation).  It is not hard to show that
such a Gibbs measure always exists (in a far more general
context, see for example \cite{Georgiibook}).

However,  unlike stationary distributions on
finite graphs, the associated Gibbs measures  on infinite graphs may
not be unique.  If there are multiple Gibbs measures associated with
a given arrival rate or activity $\lambda$, we say that there is
\emph{phase coexistence} at that $\lambda$.
Let $T_n$ denote the finite sub-tree of $\Tree^b$ with root $r$ and depth $n$,
which contains all vertices in $\Tree^b$ that are at a distance of at most $n$ from the root $r$.
As is well known (see, for example, Chapter 4 of \cite{Georgiibook}), for a fixed
activity $\lambda > 0$,
one way to obtain a Gibbs measure on the tree $\Tree^b$ rooted at $r$
is as the suitable limit of a sequence of measures,
where the $n$th
measure in the sequence is the stationary measure $\mu_{T_n \cup \partial T_n,
\lambda}$ on $T_n \cup \partial T_n$ (as defined in \eqref{formula-statdist}),
conditioned on the boundary $\partial T_n$
being empty (\ie conditioned on all vertices in the boundary having state $0$).
We shall refer to this Gibbs measure as the empty boundary condition (b.c.) Gibbs measure
(corresponding to the activity $\lambda$).
In a similar fashion,  we define the full b.c.\ Gibbs measure
to be the limit of a sequence of conditioned measures on $T_n$,
but now conditioned on  the boundary $\partial T_n$ being full (\ie conditioned on all vertices
in the boundary having
state $C$).
Let $\delta_{\lambda}$ denote the total variation distance of
the marginal distributions at the root $r$ under the empty b.c.\ and
full b.c.\ Gibbs measures corresponding to the activity  $\lambda$.
When $\lambda$ lies in the
region of uniqueness, clearly the empty b.c.\ Gibbs measure
coincides with the full b.c.\ Gibbs measure, and so $\delta_{\lambda} = 0$.
On the other hand, when $\lambda$ is in a region of
phase coexistence, then $\delta_{\lambda} > 0$
and it can be shown (due to a certain
monotonicity property of our model established in Lemma \ref{lem-mon} and
Proposition \ref{criterium}) that the empty b.c.\ and full b.c.\ Gibbs
measures must necessarily differ.
 If there exists $\lambda_{cr} = \lambda_{cr}(C)$
for which there is uniqueness for each $\lambda < \lambda_{cr}$ and
phase coexistence for every $\lambda > \lambda_{cr}$, then we say
that a \emph{phase transition} occurs at $\lambda_{cr}$.    Moreover, if
 $\delta_{\lambda}$, as a function of $\lambda$, is continuous
 at $\lambda_{cr}$, then we say that a {\em second-order}  phase
transition occurs, while if $\delta_{\lambda}$ is discontinuous at $\lambda_{cr}$, then we say that a {\em first-order} phase transition occurs.

When $C = 1$, the phase transition point $\lambda_{cr}(1)$ on the tree is
explicitly computable and is easily seen to be a second-order phase transition
(see \cite{kel,Zachary83,Zachary85} and also Section \ref{subs-rec}).
The behavior is more complicated for higher $C$.
The multi-state  hard core model on
the $d$-dimensional lattice $\bbZ^d$ was studied
in \cite{mazsuh}, where it was shown that
 when $C$ is odd, there is phase coexistence for all sufficiently large
 $\lambda$, while when $C$ is even,
there is a unique Gibbs measure for each sufficiently
large $\lambda$.
If phase coexistence were known to be monotone in the activity (this remains
an open problem on $\bbZ^d$ even when $d = 2$),
then
the  result of Mazel and Suhov would imply that there is no phase transition
on $\bbZ^d$ for even $C$.
On the other hand, numerical experiments for the multi-state
hard core model on the regular tree
 (see Section 3.5 and Figure 5 of \cite{ramzie02})
suggest that there is a  phase transition on the tree for every $C$,
but that the order of the phase transition depends on the parity of $C$ (being first-order for even $C$ and second-order for odd $C$).
This is particularly interesting as it shows that the parity of the capacity
 has an effect on the regular tree as well,
although the effect is not as pronounced as on the $d$-dimensional lattice.

The study of phase transitions of models with hard constraints on trees
has been the subject of much recent research (see \cite{briwin99,briwin02,marrozsuh04}).  In \cite{briwin99}, the focus
is on classifying types of hard constraints (as encoded in a so-called
constraint graph) on the basis of whether or not there exists
a unique simple invariant Gibbs measure for all activity vectors
$(\lambda_i, i \in S)$.
For $C  > 1$, the model that we present here allows for two $1$'s to be
adjacent, but
never allows a $1$ to be adjacent to $C$ which, in the language of
\cite{briwin99},  implies that  the associated constraint graph
is fertile. From Theorem 8.1 of \cite{briwin99} it
follows
 that there exist some activity vectors for which there exist
multiple simple invariant Gibbs measures.  However, the emphasis of our
work is quite different, as our aim is
to identify regions where multiple Gibbs measures (not necessarily
simple and invariant)  exist
for the particular choice of activity vector given in (\ref{def-lambda2}).
Another related work, again motivated by telecommunication networks,
is \cite{marrozsuh04}, which studies
Gibbs measures associated with
a three-state generalization of the hard core model.
However, the hard constraints considered in \cite{marrozsuh04}
are somewhat different from the $C = 2$ case in our model.

\subsection{Main Results and Outline}

The main contribution of this paper is to make rigorous some of
the conjectures made in \cite{ramzie02}, leading to a better understanding
of the multi-state hard core model. Our results may be broadly summarized as
follows.
\begin{enumerate}
\item For $C=2$ and every $b \in \bbN$, $b\geq 2$, we show that the Gibbs measure
is unique for
  larger values of $\l$ than in the usual $C=1$ hard core model (see Corollary
  \ref{2-1}) and we also show that the phase transition
is first-order (see Theorem \ref{C=2}).
Recall that, in contrast, for  $C=1$, the phase transition is second-order.
\item For large values of $b$, we identify
a rather narrow range of values for $\l$, above which there is phase
co-existence and below which there is uniqueness.  Although we do not establish the existence of a unique critical value $\l_{cr}(C)$ at which phase transition occurs, we establish a fairly
precise estimate of  $\lambda_{cr}(C)$ if (as we strongly believe) it exists:
when  $C$ is odd,
$$
\l_{cr}(C) \approx \left(\frac{e}{b}\right)^{\frac{1}{\ceil{C/2}}}
$$
(see Theorem \ref{critical}),  while for $C$ even,
$$
\l_{cr} (C) \approx
\left(\frac{1}{C+2}\frac{\log b}{b}\right)^{\frac{2}{C+2}}
$$
(see Theorem \ref{EVEN}).
\item For all even $C$ and all sufficiently large $b$ (depending on $C$), the model always
   exhibits a first-order phase transition (see Section \ref{EVEN2}).
\end{enumerate}

The outline of the paper is as follows.   First, in Section
\ref{sec-gibbs} we establish a connection between phase coexistence
and multiplicity of the fixed points of an associated recursion.
This is based on the construction of Gibbs measures as limits of conditional measures on finite trees with boundary conditions, as mentioned above.
In Section \ref{sec-recc2} we provide a detailed analysis of the recursion in the special
case $C=2$. In Section \ref{sec-largeb} we study the recursion when $b$ is large and identify the phase transition window. Finally, in Section
\ref{EVEN2} we study the asymptotics for large $b$ when $C$ is even and
provide evidence of a first-order phase transition.
An interesting open problem is to rigorously establish that the phase
transition is second-order for all odd $C$.

\section{Gibbs Measures and Recursions}
\label{sec-gibbs}

\subsection{Gibbs Measures on Trees}
\label{subs-gibbs}

Consider any graph $G = (V, E)$ with vertex set $V$ and
edge set $E \subseteq V^{(2)}$ (the set of unordered pairs from $V$).
For any $U \subset V$, the boundary of $U$ is
 $\partial U \doteq \{x \in V\setminus U:
xz \in E ~\mbox{for some}~ z \in U\}$ and the closure of
$U$ is $\overline{U} \doteq
U \cup \partial U$.  Let $G[U]$ denote the restriction of the graph
to the vertex set $U$.  For $\sigma \in S_C^V$, let $\sigma_U = (\sigma_v, v \in U)$
represent the projection of the configuration $\sigma$ onto the
vertex set $U$. With some abuse of notation, for conciseness, we will
write just $\sigma_v$ for $\sigma_{\{v\}}$ and refer to it as the \emph{state}
or, inspired by models in statistical mechanics, the \emph{spin value} at $v$.   For
$U \subseteq V$, let
$\cF (U)$ be the $\sigma$-field in $S_C^U$ generated by sets of
the form $\{\sigma_v = i\}$ for some $v \in U$
and $i \in S_C$.
We now provide a rigorous definition of the Gibbs measure.

\begin{definition}
\label{def-gibbs}
A Gibbs measure for the multi-state hard core model associated
with the activity $\lambda$
is a probability measure $\mu$ on $(S_C^V, \cF (V))$ that satisfies
for all $U \subset V$ and $\mu$-a.a. $\tau \in S_C^V$,
$$
\mu ( \sigma_U = \tau_U| \sigma_{V\setminus U} = \tau_{V\setminus U})
= \mu_{G[\bar{U}], \lambda} (\sigma_U = \tau_U| \sigma_{\partial U} =
\tau_{\partial U}),
$$
where $\mu_{G[\bar{U}], \lambda}$ is as defined in (\ref{formula-statdist}),
with $\lambda_i$ given as in (\ref{def-lambda2}).
\end{definition}

We now specialize to the case when $G$ is a regular, $b$-ary, rooted tree $\Tree^b$ with root $r$.
A {\it child} of a vertex $x$
in $\Tree^b$ is a neighboring vertex that is further from $r$ than $x$; the vertices (other than $x$) that lie along the path from $x$ to $r$ are the {\it ancestors} of $x$.
We will be concerned
with (complete) finite sub-trees $T$ of~$\Tree^b$ rooted at $r$; such a tree $T$ is determined by
a depth $n$, and consists of all those vertices at distance
at most $n$ from $r$.
It has $|T|=(b^{n+1}-1)/(b-1)$ vertices, and its
{\it boundary\/}~$\partial T$ consists of the children (in~$\Tree^b$)
of its leaves (so that $|\partial T|=b^{n+1}$). The tree consisting of all vertices
at distance at most $n$ from the root $r$ will be denoted by $T_n$.

Given a finite sub-tree $T$ and
$\tau\in \O_{\Tree^b}$, we let $\Omega_T^\tau$ denote
the (finite) set of  spin configurations
$\sigma\in \O_{T\cup\partial T}$ that agree with~$\tau$
on~$\partial T$; thus $\tau$ specifies a {\it boundary condition\/}
on~$T$.
For a function $f:\Omega_{T\cup \partial T}\to\bbR$ we denote by
$\mu_{T,\l}^\tau(f)= \sum_{\sigma\in\Omega_T^\t} \mu_{T,\l}^\tau(\sigma)f(\sigma)$
the {\it expectation\/} of~$f$ with respect to the
distribution~$\mu_{T,\l}^\tau(\s)\propto \prod_{v\in T}\l^{\s_v}$.
On the configuration space $\O_{\Tree^b}$, we define the partial order
$\s \prec \h$ if and only if $\s_v\le \h_v$ for all $v$ with even
 distance $d(v,r)$ from the root and $\s_v \ge \h_v$
 for all $v$ with odd distance from the root. Given two probability measures on
$\O_{\Tree^b}$, we then say that $\mu\prec \nu$ if $\mu(f)\le \nu(f)$ for
any (bounded) function $f$ that is non-decreasing with respect to the above
partial order.

Let $T$ be a
complete finite tree rooted at $r$, and let
$\mu_{T,\l}^0$ and $\mu_{T,\l}^C$, respectively,
 be the empty b.c.\ and full b.c.\
measures (corresponding to the two boundary conditions identically equal
to $0$ and $C$, respectively, on  $\partial T$).
The following monotonicity result is well known
(see, for example, Theorem 4.1 of \cite{Zachary85}).
However, for completeness, we provide an independent proof
of this result, which involves
a Markov chain argument that constructs a
simultaneous coupling of $(\mu_{T,\l}^0, \mu_{T,\l}^\tau,\mu_{T,\l}^C)$ such
that the required monotonicity
conditions are satisfied with probability one.

\begin{lemma}
\label{lem-mon}
For any $\tau \in \Omega_{\Tree^b}$,
$$
\begin{array}{rl}
\mu_{T,\l}^0\prec\mu_{T,\l}^\tau\prec \mu_{T,\l}^C &  \mbox{ if } d(\partial
T, r) \mbox{ is even, } \\
\mu_{T,\l}^C\prec\mu_{T,\l}^\tau\prec \mu_{T,\l}^0 &  \mbox{ if } d(\partial
T, r) \mbox{ is odd.}
\end{array}
$$
Moreover, if $d (\partial T, r)$ is even (respectively, odd)
there is a coupling $\pi_{T} = (\sigma^0, \sigma^\tau,
\sigma^C)$ of $(\mu_{T,\l}^0, \mu_{T,\l}^\tau,\mu_{T,\l}^C)$ such that
$\sigma^0 \prec \sigma^\tau \prec \sigma^C$ (respectively,
$\sigma^C \prec \sigma^\tau \prec \sigma^0$) with probability one.
\end{lemma}

\begin{proof}
We consider only the
case when $d(\partial
T, r)$ is even, since the other case can be established in an
exactly analogous fashion.
On $\O_T^0\times\O_T^\tau\times\O_T^C$ we  construct  an
ergodic Markov chain $\{\s^0(t),\s^\tau(t),\s^C(t)\}_{t\in \bbZ_+}$ such
that at any time $t\in  \bbZ_+$ the required ordering relation
$\s^0(t) \prec \s^\tau(t) \prec \s^C (t)$ is satisfied,
and moreover each replica is itself an ergodic chain that is
reversible with respect to the measure $\mu_{T,\l}^{\cdot}$ with the
corresponding boundary condition. The stationary
distribution $\pi_T$ of the global chain will then represent the sought
coupling of the three measures.

The chain, a standard \emph{Heat Bath} sampler, is defined as
follows. Assume that the three current configurations corresponding to
$0,\tau$ and $C$ boundary conditions are equal to $(\a,\b,\g)$
respectively and that they satisfy the ordering relation. Pick uniformly at random
$v\in T$ and let $(a,b,c)$ be the maximum spin values in $v$ compatible
with the values of $(\a,\b,\g)$ on the neighbors of $v$
respectively. Due to the ordering assumption either $c\le b\le a$ or
the opposite inequalities hold. Then the current three values at
$v$ are replaced by new ones, $(\a'_v,\b'_v,\g'_v)$, sampled from a
coupling of the three distributions on $\{0,1,\ldots,a\},\ \{0,1,\ldots,b\},\
\{0,1,\ldots,c\}$ which assign a weight proportional to $\l^i$ to the value
$i$. It is clear that such a coupling can be constructed in such a way that
$(\a'_v,\b'_v,\g'_v)$ satisfy the opposite ordering of $(a,b,c)$ and
thus the global ordering is preserved.
\end{proof}

Consider now the sequence $\{T_{2n}\}_{n\in \bbN}$  with $d(\partial T_{2n},r)=2n$. Then, thanks to
monotonicity, $\lim_{n\to \infty}\mu^{C}_{T_{2n},\l}=\mu_\l^{C}$ exists
(weakly) and it defines the \emph{maximal} Gibbs measure. Similarly
$\lim_{n\to \infty}\mu^{0}_{T_{2n},\l}=\mu_\l^{0}$ defines the \emph{minimal}
Gibbs measure \cite{Georgiibook}. Notice that, by construction, $\lim_{n\to \infty}\mu_{T_{2n+1},\l}^C=\mu_\l^0$ while
$\lim_{n\to \infty}\mu_{T_{2n+1},\l}^0=\mu_\l^C$.  Finally, for any other Gibbs measure $\mu$, it holds that
$\mu_\l^0\prec \mu_\l\prec \mu_\l^C$.

The main problem is therefore that of deciding when $\mu_\l^C=\mu_\l^0$.
In what follows we establish the following criterion, which is in fact an
equivalent criterion, since the other implication is obviously true
(see \cite{Zachary83,Zachary85}; see also \cite{vanste,spi}
for a similar discussion in the special case of $C=1$).
Let $\bbP^\tau_{n,\l}$ be the
marginal of $\mu_{T_n,\l}^\tau$ on $\s_r$ given boundary condition
$\tau$, and let $\bbP^C_\l$ and $\bbP^0_\l$
be the corresponding marginals for $\mu^C_\l$ and $\mu^0_\l$, respectively.
\begin{proposition}
\label{criterium}
For every $\lambda > 0$,
if $\bbP^C_{\l}=\bbP^0_{\l}$ then $\mu^C_\l=\mu^0_\l$\,.
\end{proposition}

\begin{proof} Assume $\bbP^C_{\l}=\bbP^0_{\l}$. Then, by monotonicity,
\begin{equation} \label{tvd}
\lim_{n\to\infty}\|\bbP^C_{n,\l}-\bbP^0_{n,\l}\|_{TV}=0,
\end{equation}
where $\|\cdot\|_{TV}$ denotes total variation distance.
Let $A$ be a
  local event (\ie depending only on finitely many spins) and let $m$ be
  sufficiently  large so that $A$ does not depend on the spin configuration outside $T_{m}$.
Fix $n>m$,
  and let $\pi_{2n} = (\sigma^0, \sigma^\tau,
\sigma^C)$ be the monotone coupling
of $(\mu_{T_{2n},\l}^0, \mu_{T_{2n},\l}^\tau,\mu_{T_{2n},\l}^C)$
 described in Lemma \ref{lem-mon}. Then
\begin{eqnarray*}
  \label{eq:2}
  \|\mu_{T_{2n},\l}^C(A)-\mu_{T_{2n},\l}^0(A)\| & \le & \pi_{2n}(\s_v^C\neq \s^0_v\ \text{for some }v\in T_{m})\\
& \le & \sum_{v\in T_m \atop d(v,r) \ even}\sum_{k=0}^C\pi_{2n}(\s^C_v\ge k>\s^0_v) \\
& & +
\sum_{v\in T_m \atop d(v,r) \ odd}\sum_{k=0}^C\pi_{2n}(\s^0_v\ge k>\s^C_v)\\
& = & \sum_{v\in T_m \atop d(v,r) \ even}\sum_{k=0}^C\bigl[ \mu_{T_{2n},\l}^C(\s_v\ge k)-\mu_{T_{2n},\l}^0(\s_v\ge k)\bigr] \\
& & + \sum_{v\in T_m \atop d(v,r) \ odd}\sum_{k=0}^C\bigl[ \mu_{T_{2n},\l}^0(\s_v\ge k)-\mu_{T_{2n},\l}^C(\s_v\ge k)\bigr]\,.
\end{eqnarray*}
For simplicity, let us examine an even term $\mu_{T_{2n},\l}^C(\s_v\ge k)-\mu_{T_{2n},\l}^0(\s_v\ge k)$ and show that it tends to zero as $n\to \infty$.
Let $w$ be the immediate ancestor of $v$. By conditioning on the spin value at $w$ we can write
\begin{eqnarray*}
 & & \mu_{T_{2n},\l}^C(\s_v=i)-\mu_{T_{2n},\l}^0(\s_v=i)\\
 & =& \sum_{j=0}^{C-i}\mu_{T_{2n},\l}^C(\s_w=j)\bigl[\mu_{T_{2n},\l}^C(\s_v=i\tc \s_w=j)-
\mu_{T_{2n},\l}^0(\s_v=i\tc \s_w=j)\bigr]\\
& & +\sum_{j=0}^{C-i}\bigr[\mu_{T_{2n},\l}^C(\s_w=j)-\mu_{T_{2n},\l}^0(\s_w=j)\bigr]\mu_{T_{2n},\l}^0(\s_v=i\tc \s_w=j)\,.
\end{eqnarray*}
By iterating upwards until we reach the root, and using \eqref{tvd}, we see that it is enough to show that
\begin{equation*}
\lim_{n\to\infty}  \max_{v\in T_m}\max _{i\le C}\max_{j\le C-i} \left|\mu_{T_{2n},\l}^C(\s_v=i\tc \s_w=j)-
\mu_{T_{2n},\l}^0(\s_v=i\tc \s_w=j)\right|=0.
\end{equation*}
Now, let
\begin{equation}
\label{Z(i)}
Z_k^\tau(i) := \l^i\sum_{\s\in \O_{T_k\setminus \{r\}}^\tau}\prod_{v\in T_k\setminus \{r\}}\l^{\s_v}
\end{equation}
denote the partition function (or normalizing constant) on the complete finite
tree $T_k$ with boundary conditions $\tau$ and $\s_r=i$.  It is clear that
\begin{equation*}
  \frac{\bbP^\tau_{k,\l}(i)}{\bbP_{k,\l}^\tau(0)}=\frac{Z_k^\tau(i)}{Z_k^\tau(0)}\,.
\end{equation*}
Therefore,
\begin{gather*}
  \mu_{T_{2n},\l}^0(\s_v=i\tc \s_w=j)=\frac{Z_{2n-n_v}^0(i)}{\sum_{k\le C-j}Z_{2n-n_v}^0(k)}
=\frac{\bbP^0_{2n-n_v}(i)}{\bbP^0_{2n-n_v}(\s_r\le C-j)}\,,
\end{gather*}
where $n_v$ denotes the level of $v$ (the distance from the root).
A similar relation holds
 for the full boundary condition.

The proof is concluded once we observe that $n_v\le m$ and that
$$
\bbP^0_{2n-n_v,\l}(\s_r\le C-j)\ge \bbP^0_{1,\l}(0)>0.
$$
\end{proof}

\subsection{Recursions}
\label{subs-rec}

Our next step, as in many other spin models on trees,
is to set up a recursive
scheme to compute the relevant marginals $\bbP^{0}_{n,\lambda}$
and $\bbP^{C}_{n,\lambda}$.
In what follows, for simplicity we count the levels bottom-up and the
boundary conditions are at level $0$. Moreover, since the recursive scheme is
independent of the boundary conditions, and since we will never be considering
more than one value of $\l$ at a time, we drop both from our notation.

For $i = 0, \ldots, C$, and $n \in \bbN$, we set
\begin{equation*}
 Q_n(i):=  \frac{\bbP_n(i)}{\bbP_n(0)},\quad
R_n(i):= \frac{\sum_{k=0}^C Q_n(k)}{\sum^{C-i}_{k=0} Q_n(k)}=
\left[1-\bbP\bigl(\s_r> C-i\bigr)\right]^{-1}\,.
\end{equation*}
Thus $R_n(0)=1$ and  $R_n(i)\le R_n(i+1)$. Moreover, let $Z_n$ be
as defined in \eqref{Z(i)}, but with $\tau$ equal to the empty b.c.\.
Then we obtain the recursive equations
\begin{eqnarray}
Z_{n+1}(i) & = & \l^i \,\left[\sum_{k=0}^{C-i}Z_n(k)\right]^b, \nonumber \\
Q_{n+1}(i) & = & \l^i\,
  \left[\frac{\sum_{k=0}^{C-i}Q_n(k)}{\sum_{k=0}^{C}Q_n(k)}\right]^b =
  \frac{\l^i}{R_n^b(i)}, \nonumber \\
R_{n+1}(i) & = &
\frac{\sum_{k=0}^{C}\frac{\l^k}{R^b_n(k)}}{\sum_{k=0}^{C-i}\frac{\l^k}{R^b_n(k)}}\,.
  \label{rec:3}
\end{eqnarray}

The case when $C = 1$ (the usual hard core model) can therefore be studied
by analyzing a one-dimensional
 recursion governed by the following maps:
\begin{equation}
\label{jmaps}
 J(x):= \frac{\l}{(1+x)^b},\qquad   J_2(x):=
J(J(x))=\frac{\l}{(1+\frac{\l}{(1+x)^b})^b}.
\end{equation}
Indeed, $J$ defines the recursion for the quantity $Z_{n}(1)/Z_{n}(0)$,
while $J_2$ defines the recursion of this quantity
between two levels on the tree.
We close this section with a summary of the properties
of $J$ and $J_2$ which, when combined with Proposition \ref{criterium},
show  that $\l_{cr}(1) :=
  b^b/(b-1)^{b+1}$ is the phase transition point for the
standard hard core model (see, for example,  \cite{kel}), and that
the phase transition for $C=1$ is second-order.
These properties will turn out to also
 be useful for our analysis of the
higher-dimensional recursions (\ie when $C \geq 2$).
We start with the definition of an $S$-shaped
function.

\begin{definition}
\label{def-sshape}
A twice continuously differentiable
function $f:[0,\infty) \mapsto [0,\infty)$ is said to be
{\em $S$-shaped} if it has the following properties:
\begin{enumerate}
\item
it is  increasing on $[0,\infty)$ with
$f(0) > 0$ and $\sup_{x} f(x) < \infty$;
\item
there exists
$\overline{x} \in (0,\infty)$ such that the derivative $f'$ is monotone
increasing in the interval $(0,\overline{x})$ and monotone
decreasing in the interval $(\overline{x}, \infty)$; in other
words, $\overline{x}$ satisfies $f''(\overline{x}) = 0$ and
is the unique inflection point of $f$.
\end{enumerate}
\end{definition}

For future purpose, we observe here
that the definition immediately implies
that for any $\theta > 0$, and $S$-shaped function $f$,
 $\theta f$ is also an $S$-shaped function.
It is also  easy to verify that
any $S$-shaped function has at most three fixed points in $[0,\infty)$,
\ie points $x \in (0,\infty)$ such that $f(x) = x$.
   We now
summarize the salient properties of $J_2$ (see e.g. Fig.~\ref{fig:J2}), all of which may easily be verified with some calculus.
\begin{enumerate}
\item
$J_2$ is an $S$-shaped function with
  $J_2(0)=\lambda/(1+\lambda)^b$ and $\sup_x J_2(x)=\lambda$,
and a unique point of inflection $x_* \in (0,\infty)$.
\item $J$ has a unique fixed point, $x_0$, which is
also a fixed point of~$J_2$.
\item  If~$\l\le \l_{cr}(1)$
then $J_2'(x)\le 1$ for any $x\ge 0$ and $x_0$ is
  the unique fixed point of~$J_2$.
\item   If $\l > \l_{cr}(1)$,
then $J_2$ has three fixed points $x_- <
  x_0 < x_+$, where $J(x_-)=x_+$ and $J(x_+)=x_-$. Moreover $J_2'(x_0)>
1$, $J_2'(x)<1$ for $x\in [0,x_-]\cup [x_+,+\infty)$ and the three fixed
  points converge to $x_0(\l_{cr}(1))$ as $\l\downarrow \l_{cr}(1)$.
\end{enumerate}

\begin{figure}
\centering
\includegraphics[angle=-90, width=4in]{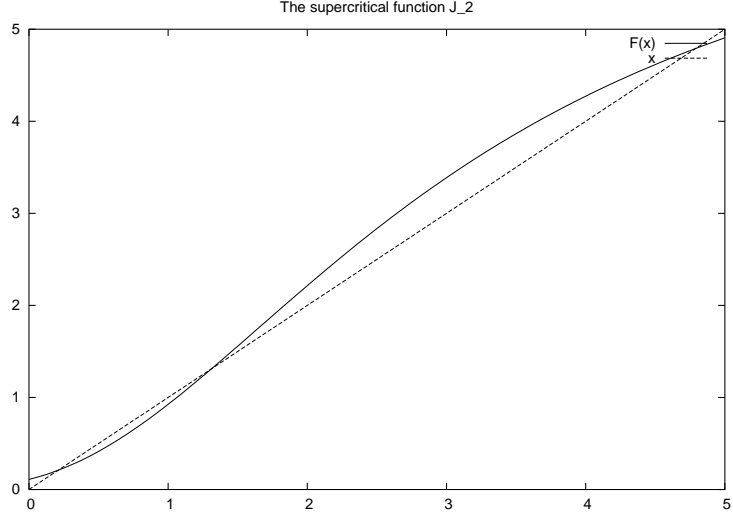}
\caption{Graph
 of the function $J_2(x)$ for $b=2,\ \lambda=7$ ($\l_{cr}=4$) }
 \label{fig:J2}
 \end{figure}

\section{Analysis of the recursions when  $C=2$} \label{sec-recc2}
When $C=2$ we have
$R_n(1)=\bigl[1-\bbP_n\bigl(\sigma_r  = 2 \bigr)\bigr]^{-1}$
and (\ref{rec:3}) can be written as:
\begin{eqnarray}
\nonumber
  R_n(0)&= & 1, \\
\label{rec-r2}
R_{n+1}(2) &= & 1+ \frac{\l}{R_n^b(1)} + \frac{\l^2}{R_n^b(2)}, \\
\nonumber
R_{n+1}(1) &= &  \frac{1+ \frac{\l}{R_n^b(1)} +
    \frac{\l^2}{R_n^b(2)}}{1+\frac{\l}{R_n^b(1)}} \quad =
\frac{R_{n+1}(2)}{1+\frac{\l}{R_n^b(1)}}. \,
  \end{eqnarray}
On replacing $n$ by $n-1$  in the last equation above, we see that
$$
R_{n}(2)=R_{n}(1)\left(1+\frac{\l}{R_{n-1}^b(1)}\right)\,.
$$
Substituting this back into \eqref{rec-r2}, we
obtain an exact two-step recursion for $Y_n:=R_n(1)$:
\begin{eqnarray}
  \label{eq:main_recursion}
  Y_{n+1} & = & 1+ \frac{\l^2}
{\left[1+\frac{\l}{Y_n^b}\right]\left[Y_n(1+\frac{\l}{Y_{n-1}^b})\right]^b}\\
\nonumber
& = & 1+ \frac{\l^2}
{\left[Y_n^b+ \l\right]\left[1+\frac{\l}{Y_{n-1}^b}\right]^b}\,.
\end{eqnarray}
It is useful to determine the initial conditions $(Y_0,Y_1)$ for the
recursion given the boundary conditions at the $0^{\rm th}$ level.
\begin{equation*}
  (Y_0,Y_1)=
  \begin{cases}
  (+\infty,1) & \text{ if the b.c.\ is full (\ie identically $C$)}\\
  (1,1+\frac{\l^2}{1+\l})  & \text{ if the b.c.\ is empty (\ie identically $0$)}\,.
  \end{cases}
\end{equation*}
Numerical calculations of
 (\ref{eq:main_recursion}) using Mathematica
strongly suggest that the critical value $\l_{cr}$, below which the recursion
 settles to a limit independent of the initial values, takes
approximately the following values:
\bigno
 \begin{center}
 \begin{tabular}{|c|c|} \hline
 {$b$} & {\hskip 1cm $\l_{cr}$} \\ \hline
 $2$ & \hskip 1cm $7.2753875$ \\ \hline
$3$  & \hskip 1cm $3.58029$ \\ \hline
$10$  & \hskip 1cm $1.107665$\\ \hline
$100$  & \hskip 1cm $0.2817409$\\ \hline
 \end{tabular}
 \end{center}
\bigno
and  that the transition is always first order (\ie if $\limsup_n Y_n\neq
\liminf_n Y_n$ then their difference is strictly larger than  some positive
constant $\d$).  Similar observations were made in \cite{ramzie02} (see Section
3.4 therein). Here, we provide a rigorous proof of these results.

Let us change variables from $Y_n$ to $X_n:=Y_n-1$ in
 (\ref{eq:main_recursion}).  It then follows that
\begin{align}
\label{Fp}  X_{n+1}&\le \frac{\l^2}
{\left[\min_{j\ge n}Y_j^b+ \l\right]\left[1+\frac{\l}{(1+X_{n-1})^b}\right]^b}
  \equiv F^{(n)}_+(X_{n-1}), \\
  X_{n+1}&\ge \frac{\l^2}
{\left[\max_{j\ge n}Y_j^b +
  \l\right]\left[1+\frac{\l}{(1+X_{n-1})^b}\right]^b}\equiv F^{(n)}_-(X_{n-1}).
\label{Fm}
\end{align}
The maps $F_\pm^{(n)}$ defined above can be rewritten in terms of the
map $J_2$ defined in \eqref{jmaps} as follows:
$$
\begin{array}{rcl}
\displaystyle  F^{(n)}_-(x) & = & \displaystyle
\frac{\l}{\left(\max_{j\ge n}Y_j^b +\l\right)}J_2(x); \\
 \displaystyle  F^{(n)}_+(x) & = & \displaystyle
\frac{\l}{\left(\min_{j\ge n}Y_j^b +\l\right)}J_2(x).
\end{array}
$$
Next, for $\kappa \geq 0$, we define
\begin{equation}
\label{fk}
 F_\kappa(x):=\frac{\l}{\kappa   +\l}J_2(x),
\end{equation}
so that $F_0 = J_2$.
For any $\kappa \geq 0$, $F_\kappa$ is a strictly positive
multiple of $J_2$ and hence is also an $S$-shaped function
(with the same inflection point $x_*$).
If we denote the fixed points of $F_\kappa$ by~\hbox{$
x_{-}^{(\kappa)}\le x_0^{(\kappa)}\le x_{+}^{(\kappa)}$} (with the obvious
  meaning) we see that:
\begin{enumerate}
 \item if $F_\kappa$ has a unique fixed point $x_0^{(\kappa)}$ then
  necessarily $x_0^{(\kappa)}< \min(x_-,x_0)$;
\item since $F'_\kappa(x)=\frac{\l}{\kappa +\l}J'_2(x)$ necessarily
  $F'_\kappa(x)\le 1$ for $x\le x_-^{(\kappa)}$;
\item the critical value $\l_c(\kappa)$ of $\l$ such that $F_\kappa$ starts to have
  three fixed points is increasing in $\kappa$. In particular,
\[ \l_c(\kappa)> \l_c(0)= \l_{cr}(1) = \frac{b^b}{(b-1)^{b+1}}; \]
\item  if $F_\kappa$ has three fixed points then
  necessarily $x_{-}^{(\kappa)}< x_-$ and $x_0 <
  x_0^{(\kappa)}<x_{+}^{(\kappa)}$;
\item the smallest fixed point $x_-^{(\kappa)}$ is continuously differentiable
in $\kappa>0$. Indeed, by the implicit function theorem and the fact that
$F_\kappa' (x_-^{(\kappa)}) < 1$, it follows that
\begin{eqnarray*}
  \label{eq:d2}
  \frac{d}{d\kappa}x_-^{(\kappa)} & = &
-\dfrac{\frac{\partial}{\partial\kappa} F_\kappa(x_-^{(\kappa)})}{
\frac{\partial }{ \partial x}F_\kappa(x_-^{(\kappa)}) - 1} \\
& = & -\dfrac{\frac{1}{\lambda + \kappa}F_\kappa(x_-^{(\kappa)})}{
\frac{\lambda}{\lambda + \kappa} J_2'(x_-^{(\kappa)}) - 1} \\
& = & -\frac{x_-^{(\kappa)}}{\l(1-J_2'(x_-^{(\kappa)}))+\kappa}.
\end{eqnarray*}
\end{enumerate}

In what follows, let
\begin{equation}
\label{mM}
 m:=\liminf_n X_n \quad \mbox{ and } \quad M:=\limsup_n X_n.
\end{equation}
We are now ready to prove our first result.
\begin{proposition} \label{prop-unique-fixed-point}
Assume that $\l > 0$ is such that $F_{1}$ has a unique fixed point. Then $M=m$
and hence the
recursion (\ref{eq:main_recursion}) has a unique fixed point.
\end{proposition}

\begin{proof}
Since $Y_n\ge 1$, it follows from (\ref{Fp}) and \eqref{fk} that
$X_{n+1}\le F_{1}(X_{n-1})$.  Since $F_{1}$ is $S$=shaped and is
assumed to  have a unique fixed point, this implies that
 $M \le x_0^{(1)}$.  Moreover, recalling that $m=\liminf_n X_n$, we see that
for any $\eps>0$, $X_n\ge m+\eps$
for all $n$ large enough. Hence, (\ref{Fp}) and \eqref{fk} imply that for
all large enough $n$, $X_{n+1}\le F_{\kappa}(X_{n-1})$ with
$\kappa =(1+m+\eps)^b$.  Thus we obtain
\begin{equation} \label{int1}
M \in (0,x_-^{(1+m+\eps)^b}).
\end{equation}
Indeed, if $F_\kappa$ has a unique fixed point, then (\ref{int1}) follows immediately. On the other hand, if
$F_\kappa$ has three fixed points then we immediately have
$M \in (0,x_-^{\kappa}) \cup (x_0^\kappa, x_+^\kappa)$. But $M \le x_0^{(1)}$ and so
$M < x_0$ (by property (1) of the $F_\kappa$'s) and also in this case $x_0 < x_0^{(\kappa)}$ (by property (4) of the $F_\kappa$'s), giving (\ref{int1}).

Since $\eps > 0$ is arbitrary in (\ref{int1}), we have in fact $M\le
x_-^{(1+m)^b}$. Similarly, using (\ref{Fm})  and \eqref{fk},
we see that $m\ge x_-^{(1+M)^b}$.
We want to conclude that necessarily $m=M$.  We write
\begin{equation*}
  M-m \le \int_m^M ds\ \left(-\frac{d}{ds} x_-^{((1+s)^b)}\right)\,,
\end{equation*}
and the sought statement will follow if, for example,
\begin{equation*}
  \label{eq:d1}
\sup_{m\le s\le M}\left|\frac{d}{ds} x_-^{((1+s)^b)}\right|<1\, .
\end{equation*}
By properties (1) and (4) of $F_\kappa$ it follows that $x_-^{(\kappa)} < x_0$, and
hence  property (4) of $J_2$ implies
 $J_2'(x_-^{(\kappa)})\le 1$.  When combined with the expression
for $d x_-^{(\kappa)}/d\kappa$ given in property
(5) of $F_\kappa$, this implies that
\begin{equation*}
  \label{eq:d4}
  \left|\frac{d}{d\kappa}x_-^{(\kappa)}\right|\le
  \frac{x_-^{(\kappa)}}{\kappa}\,,
\end{equation*}
and hence that
\begin{equation*}
  \label{eq:d5}
  \sup_{m\le s\le M}\left|\frac{d}{ds} x_-^{((1+s)^b)}\right|\le
\sup_{m\le s\le M} b\,\frac{x_-^{((1+s)^b)}}{1+s}\le b\, x_-^{(1)}\,
\end{equation*}
where the last inequality uses the fact that $x_-^{(\kappa)} < x_-^{(1)}$ for
any $\kappa > 0$.

Thus we have to show that $x_-^{(1)}<1/b$. For this purpose it is
enough to show that $F_1(1/b)< 1/b$. We compute
\begin{gather}
  \label{eq:x1}
  b\,F_1(1/b)=\frac{\l}{(1+\l)}\frac{b\l}{\left(1+\frac{\l}{(1+\frac 1b)^b}\right)^b}\,.
\end{gather}
Next, we observe that the map $\l \mapsto b\l/(1+\frac{\l}{(1+\frac
    1b)^b})^b$ achieves its maximum at $\l_{\rm
  max}=(1+1/b)^b/(b-1)$, where it is equal to
$\frac{b}{b-1}\left[\frac{b^2-1}{b^2}\right]^b$. The latter expression is
decreasing in $b$ for $b\ge 2$ and for $b=2$ it is equal to $\frac{18}{16}$.
Therefore, if $\l/(\l+1)<\frac{16}{18}$, \ie $\l<8$, then the r.h.s.
of (\ref{eq:x1}) is strictly less than one.
\medno
We now examine the case $\l\ge 8$. We write
\begin{equation*}
  \frac{b\l}{(1+\frac{\l}{(1+\frac 1b)^b})^b} \le \frac{\l
  b}{(1+\frac{\l}{\nep{}})^b} \le \frac{b\,
  \nep{b}}{\l^{b-1}}\le\frac{b\, \nep{b}}{8^{b-1}} < 1, \quad \text{for
  $b\ge 3$}\,.
\end{equation*}
Finally the case $b=2$ and $\l\ge 8$ is handled directly:
$$
 \frac{2\l}{(1+\frac{\l}{(1+\frac 12)^2})^2}=
 \frac{2\l}{(1+\frac{4\l}{9})^2}\le
 \frac{16}{(1+\frac{32}{9})^2}\approx 0.77\,.
$$
\end{proof}

Notice that in the proof of the inequality $x_-^{(1)}< 1/b$ we did not use the
hypothesis that $F_1$ has only one fixed point. Moreover, we proved
something slightly stronger, namely
\begin{equation} \label{rem-ve}
\mbox{there exists $\eps(b)>0$ such
that $1/b - x_-^{(1)}\ge \eps(b)$ for \emph{any} $\l$}.
\end{equation}

The following monotonicity property is an immediate consequence of
Proposition \ref{prop-unique-fixed-point}.
Recall that $\l_{cr}(1) = b^b/(b-1)^{b+1}$ is the phase transition
point for the usual ($C=1$) hard core model.

\begin{corollary}
\label{2-1}
For every $\l\le \l_{cr}(1)$,
 the $C=2$ multi-state hard core model has a unique Gibbs measure.
\end{corollary}
\begin{proof}
If $J_2$ has only one fixed point then the same is true of $F_1$.
By Proposition \ref{prop-unique-fixed-point} there is then only one
fixed point for the recursion (\ref{eq:main_recursion}).  The result
then follows from Lemma \ref{lem-mon} and Proposition \ref{criterium}.
\end{proof}

The next result shows that the phase transition for $C = 2$ is first order.
Recall the definitions
of $M$ and $m$ given in (\ref{mM}) and let $\eps (b)$ be as in
\eqref{rem-ve}.
\begin{theorem}
\label{C=2}
If $m \neq M$ then $M-m> \eps(b)>0$.
\end{theorem}
\begin{proof}
Suppose $m \neq M$.  From Proposition~\ref{prop-unique-fixed-point}, it then
follows that $F_1$ (and {\em a
fortiori} $J_2$) has three
fixed points $x_-^{(1)}<x_0^{(1)}<x_+^{(1)}$, with $x_0^{(1)}> x_0$.
We now show that $x_0 > 1/b$. Indeed, since $J(x_0)=x_0$ and $J$
is strictly decreasing, it is enough to check that $J(1/b)> 1/b$ or,
equivalently, that  $\l/(1+\frac 1b)^b >1/b$.
But $\l>\frac{b^{b}}{(b-1)^{b+1}}$ and clearly
$$
\frac{b^b}{(b-1)^{b+1}(1+\frac{1}{b})^b}=\frac{b^{2b}}{\left(b^2-1\right)^b(b-1)} > \frac 1b\,.
$$
Since $1/b - x_-^{(1)} \geq \eps (b)$ by \eqref{rem-ve},
this implies $x_0^{(1)}-x_-^{(1)}>\eps(b)$.

Next, since $X_n= \left[1-\bbP_n\left(\s_r = 2\right)\right]^{-1}-1$, we infer that $X_n$
is maximized by the empty b.c.\ and minimized by the full b.c.\ if $n$ is
odd (and vice versa if
$n$ is even). Thus,
using the recursive inequality $X_{n+1}\le
F_1(X_{n-1})$, we obtain for any odd $n$, the inequality
$X_n\le U_n$, where  $\{U_n, n \mbox{ odd} \}$ is the sequence
that satisfies the recursion $U_{n+2}=F_1(U_n)$, with
$U_1 = 0$.
In particular, $m\le x_-^{(1)}\le \frac{1}{b}
-\eps(b)$. If now $M\le m+\eps(b)<x_0^{(1)}$ then necessarily $X_n <
  x_0^{(1)}$ for any $n$ large enough and repeated iterations of $X_{n+1}\le
F_1(X_{n-1})$ imply $M\le x_-^{(1)}$. At this stage we are back in the
  framework of the proof of  Proposition \ref{prop-unique-fixed-point} and $m=M$, resulting in  a contradiction.
\end{proof}

\section{The Large $b$ Asymptotic Regime} \label{sec-largeb}
In this section we set up and then analyze the recursion for any value of $C$
when $b$ is large. In what follows, $e = \exp(1)$.

For any $j\le C$ set $j^*=C-j$. Also, for $\l < 1$, set
$A_\l =\sum_{i=0}^\infty \l^i= (1-\l)^{-1}$. Iterating (\ref{rec:3}) we obtain
$$
R_{n+2}(j)=1 +\frac{\displaystyle\sum_{i=j^*+1}^C
  \l^i\left(\displaystyle\sum_{k=0}^{i^*}\frac{\l^k}{R_n^b(k)}\right)^b}
{\displaystyle\sum^{j^*}_{i=0} \l^i\left(\displaystyle\sum_{k=0}^{i^*}\frac{\l^k}{R_n^b(k)}\right)^b}.
$$
In turn, this implies that
\begin{eqnarray*}
  \label{eq:R2}
R_{n+2}(j) & \le & 1 +\frac{A_\l \l^{j^*+1}\left(\displaystyle\sum_{k=0}^{j-1}\frac{\l^k}{R_n^b(k)}\right)^b}
{\left(\displaystyle\sum_{k=0}^{j-1}\frac{\l^k}{R_n^b(k)}
+\sum_{k=j}^{C}\frac{\l^k}{R_n^b(k)}\right)^b} \\
& = &  1 +\frac{A_{\l}\l^{j^*+1}}
{\displaystyle \left(1 +
  \frac{\sum_{k=j}^{C}\frac{\l^k}{R_n^b(k)}}{\sum_{k=0}^{j-1}\frac{\l^k}{R_n^b(k)}}\right)^b}\\
& \le & 1 +\frac{A_\l \l^{j^*+1}}
{\left(\displaystyle 1 + A_\l^{-1}\frac{\l^j}{R_n^b(j)}\right)^b}\,.
\end{eqnarray*}
Therefore, by letting $X_n(j)=R_n(j)-1$ we have
\begin{equation}
  \label{eq:R3}
  X_{n+2}(j)\le A^2_\l \l^{j^*-j+1}J_2^{(\l_j)}(X_n)\equiv F^{(j)}_+(X_n(j))\,,
\end{equation}
where $\l_j:=A_\l^{-1}\l^j$, and
 $J^{(\l)} = J$ and $J_2^{(\l)} = J_2$ are the maps
defined in (\ref{jmaps}), but with the $\l$ dependence now denoted
explicitly.

In a similar fashion, we obtain a lower bound
\begin{eqnarray*}
  \label{eq:R2bis}
R_{n+2}(j) & \ge & 1 +\frac{\l^{j^*+1}\left(\displaystyle\sum_{k=0}^{j-1}\frac{\l^k}{R_n^b(k)}\right)^b}
{A_\l\left(\displaystyle\sum_{k=0}^{j-1}\frac{\l^k}{R_n^b(k)}
+\sum_{k=j}^{C}\frac{\l^k}{R_n^b(k)}\right)} \\
& = &  1 +\frac{A_\l^{-1}\l^{j^*+1}}
{\displaystyle \left(1 +
  \frac{\sum_{k=j}^{C}\frac{\l^k}{R_n^b(k)}}{\sum_{k=0}^{j-1}\frac{\l^k}{R_n^b(k)}}\right)^b}\\
& \ge & 1 +\frac{A_\l^{-1}\l^{j^*+1}}
{\left(\displaystyle 1 + A_\l\frac{\l^j}{R_n^b(j)}\right)^b}.
\end{eqnarray*}
Therefore, we have
\begin{equation}
  \label{eq:R3bis}
  X_{n+2}(j)\ge A_\l^{-2}\l^{j^*-j+1}J_2^{(\l'_j)}(X_n)\equiv F^{(j)}_-(X_n(j))\,,
\end{equation}
where  $\l'_j:=A_\l\l^j$.

\subsection{The case of $C$ odd}
We start by stating the main result of the section.
Recall that for $\lambda < 1$, $A_\l = (1-\lambda)^{-1}$.
\begin{theorem}
\label{critical}
Let $j_c=\ceil{\frac C2}$, and define $\lm := A_\l^{-1}\l^{j_c}$ and
$\lp := A_\l \l^{j_c}$. Then the following two properties hold:
\begin{enumerate}
\item
If
$\left(\frac{\g}{b}\right)^{\frac {1}{j_c}}\le \l<1$ with
$\g>\nep{\phantom{1}}$, then, for any $b$ large enough depending on $\g$,
the smallest fixed
point of
\begin{equation}
  \label{eq:A}
x \mapsto A_\l^{2}J_2^{(\lm)}(x)
\end{equation}
is strictly smaller than the largest fixed point of
\begin{equation}
  \label{eq:B}
x \mapsto A_\l^{-2}J_2^{(\lp)}(x)\,.
\end{equation}
In particular, there is phase coexistence.
\item
On the other hand, if $\l\le \left(\frac{\g'}{b}\right)^{\frac {1}{j_c}}$
with $\g'<\nep{}$ then, for
every $b$ large enough, depending on $\g'$, there is a unique Gibbs measure.
\end{enumerate}
\end{theorem}

We start by establishing the first assertion of the theorem.  Our
proof will make use of the following elementary observation.
\begin{lemma}
\label{hmap}
    For $\g>0$ the function  $H_\g:[0,\infty)\mapsto [0,\infty)$ defined by
\[ H_\g(z)=\g\,\nep{-\g\nep{-z}}, \quad z \in [0,\infty),
\]
 is $S$-shaped. In addition, the following two properties hold:
\begin{enumerate}
\item if $\g\le \nep{}$ then $H_\g$ has one
    fixed point $z_0 < 1$;
\item if $\g>\nep{}$ then
$H_\g$ has three distinct fixed points $z_-<z_0<z_+$ that satisfy
\begin{equation}
\label{tx}
0\le z_-\le
\log(\g)-\log(\log(\g))<z_0\le \log \g<z_+\,.
\end{equation}
\end{enumerate}
\end{lemma}

\begin{proof}
The function $H_\g$ is clearly twice continuously differentiable, satisfies
$H_\g(0) = \g > 0$ and $\sup_{x} H_\g (x) = \g e^{-\g}< \infty$.
That it is $S$-shaped therefore follows from the fact that
$$
H_\g'(z)=\g\nep{-z}H_\g(z)>0\,\quad \text{and} \quad
H_\g''(z)=\g\nep{-z}H_\g(z)\bigl[\g\nep{-z}-1\bigr]\,.
$$
Now suppose $\g<\nep{}$. Then $\sup_z H_\g'(z)<1$ and therefore there
exists a unique fixed point $z_0$. The fact that $z_0 <1$ follows
from the observation that
$$
H_\g(1)=\g\nep{-\g\nep{-1}}<1\,.
$$
On the other hand, if $\g=\nep{}$ the value
$z_0=\log \g$ is the unique fixed point, and satisfies $H_\g'(z_0)=1$.
Lastly, for $\g>\nep{}$, we have the inequalities
\begin{eqnarray*}
H_\g'(\log \g) & > & 1, \\
H_\g(\log \g) & > & \log \g, \\
H_\g\bigl(\log \g-\log(\log \g)\bigr) & < & \log \g-\log(\log \g),
\end{eqnarray*}
where the last inequality  holds because
$H_\g\bigl(\log \g-\log(\log \g)\bigr)=1$
and $\g\mapsto \log \g-\log(\log \g)$ restricted to the interval
$[\nep{},\infty)$ is increasing with $\log(\nep{})-\log(\log(\nep{}))=1$.
Together with the $S$-shaped property of $H$, these inequalities
immediately imply that $H$ has three fixed points that satisfy
(\ref{tx}).
\end{proof}

We are now ready to establish the first statement of Theorem \ref{critical}.

\begin{proof} [Proof of Theorem \ref{critical}(1)]
Fix $\l \in
  [\left(\frac{\g}{b}\right)^{\frac{1}{j_c}}, 1)$ with
  $\g>\nep{\phantom{1}},$ and for notational conciseness, denote
$A_{\lambda}$ simply by $A$.  We first show that the asserted
inequality between the fixed points of the two maps implies
phase coexistence.  This is a simple consequence of the fact that,
for any boundary condition $\tau$,  the sequence $\{X_n^*\}$ defined by
\[ X_n^* \equiv X_n(j_c)=\mu^\tau_{T_n}(\s_r\ge j_c)/\mu^\tau_{T_n}(\s_r\le
  j_c), \quad n \in \bbN, \]
  obeys the recurrence
$$
  A^{-2}J_2^{(\lp)}(X_n)\le X_{n+2}^*\le A^2 J_2^{(\lm)}(X_n^*)\,
$$
where we have made use of (\ref{eq:R3}) and (\ref{eq:R3bis}), together with
the duality property $j_c^*+1=j_c$.
If now $\inte{\frac C2}$ boundary
conditions are imposed at the zeroth level then $X^*_0=0$
and $X_n^*$ will always be smaller than the smallest fixed point of $ x
\mapsto A^2J_2^{(\lm)}(x) $. On the other hand, under $\ceil{\frac C2}$
boundary conditions, $X_0^*=1$ and $X_n^*$ will always be larger than the
largest fixed point of $ x \mapsto A^{-2}J_2^{(\lp)}(x) $ because the
range of this mapping is contained in $[0,1]$ for large $b$.

We now prove our statement concerning the fixed points of (\ref{eq:A}),
(\ref{eq:B}).  First, consider the case
$\l=\left(\frac{\g}{b}\right)^{\frac {1}{j_c}}$
and observe that for any $z>0$\,,
\begin{equation}
  \label{eq:limit}
\lim_{b\to \infty}b\,A^{-2}J_2^{(\lp)}(z/b)=\lim_{b\to \infty}b\,A^{2}J_2^{(\lm)}(z/b)=H_\g(z),
\end{equation}
uniformly on bounded intervals. Next, we define
$$
\tilde{x}_-:=\frac{\log \g-\log(\log \g)}{b}\,\qquad \text{and}\qquad \tilde{x}_+:=\frac{\log \g}{b}\,.
$$
From  Lemma \ref{hmap}, it follows that $H_\g (b\tilde{x}_-) < b \tilde{x}_- <
 b \tilde{x}_+ < H_\g (b \tilde{x}_+)$.  Together with
(\ref{eq:limit}), this shows that for any $b$ large enough,
$$
  A^{2}J_2^{(\lm)}(\tilde{x}_-)< \tilde{x}_- < \tilde{x}_+<A^{-2}J_2^{(\lp)}(\tilde{x}_+)\,,
$$
and the first assertion of the lemma follows (for this case) because $A^{-2}J_2^{(\lp)}$ and
$A^{2}J_2^{(\lm)}$ are $S$-shaped exactly like $H_\g$.

We now consider the
case $\left(\frac{\g}{b}\right)^{\frac {1}{j_c}}\le \l<1$ and again we
compute
  \begin{equation}
    \label{eq:rhs}
 A^{2}J_2^{(\lm)}(\tilde{x}_-)\le  A^2\frac{\lm}{\left(1+\lm\nep{-b\tilde{x}_-}\right)^b}=
A^2\frac{\lm}{\left(1+\lm\frac{\log \g}{\g}\right)^b}\,.
  \end{equation}
If $\l$ does not tend to zero as $b\to \infty$, then it is obvious that
the r.h.s of (\ref{eq:rhs}) is smaller than $\tilde{x}_-$ for large enough
$b$. If instead $\lim_{b\to \infty}\l=0$ we proceed as follows. The function
$f_\g(\l)=\l/\left(1+\l\frac{\log \g}{\g}\right)^b$ satisfies
$$
f_\g'(\l)=\frac{1}{\left(1+\l\frac{\log \g}{\g}\right)^{2b}}\left(1-\frac{b\l\log \g}{\g+\l\log \g}\right)\,,
$$
and hence is decreasing in the interval
$(\frac{\g}{(b-1)\log \g},\infty)$. Since
$\g>\nep{}$ and our assumption $\lambda \rightarrow 0$ implies  $A = A_\lambda
\approx 1$ for large
$b$, we have the inequality
\[ \lm > A^{-1}\g/b> \g/((b-1)\log \g). \]
Thus,
we can conclude that the r.h.s of (\ref{eq:rhs}) is smaller than
the same expression with $\lm$ replaced by $A^{-1}\g/b$.  After this
replacement,
the resulting r.h.s of (\ref{eq:rhs}) is indeed smaller than
$\tilde{x}_-$ for all large enough $b$ because of (\ref{tx}) and (\ref{eq:limit}).
In conclusion, we have shown that for any
$\left(\frac{\g}{b}\right)^{\frac {1}{j_c}}\le \l<1$ the function
$A^{2}J_2^{(\lm)}$ has a fixed point smaller than $\tilde{x}_-$.

Next, we examine $A^{-2}J_2^{(\lp)}$. If $\lim_{b\to
  \infty}b\lp=\infty$ then it easily follows that for large $b$, we have
$A^{-2}J_2^{(\lp)}(A^{-2}\lp/2)>A^{-2}\lp/2 \gg \tilde{x}_-$. If instead
$\lp\le C/b$ for some finite constant $C$, we choose
$x_\l=\log(b\lp)/b >x_-$ and write
\begin{equation*}
  A^{-2}J_2^{(\lp)}(x_\l)\ge A^{-2}\lp\nep{-b\lp/(1+x_\l)^b}\,.
\end{equation*}
By construction, $\lim_{b\to
  \infty}\nep{-b\lp/(1+x_\l)^b}=\nep{-1}$. Therefore, for sufficiently large $b$,
$$
A^{-2}\lp\,\nep{-b\lp/(1+x_\l)^b}\ge (1-O(b^{-1}))\lp\,\nep{-1}\ge x_\l\,,
$$
because $\lp>\g/b$ with $\g>\nep{}$. In conclusion
$A^{-2}J_2^{(\lp)}(x)$ has a fixed point strictly bigger than $x_-$ and
the existence of a phase transition follows.
\end{proof}

\bigno
We now turn to the proof of the second assertion of Theorem \ref{critical}, namely the absence of a phase transition for $\l\le
\left(\frac{\g'}{b}\right)^{\frac{1}{j_c}}$, with $\g'<\nep{}$.
For this, we first establish two preliminary results in
Lemmas \ref{vardist} and \ref{vardist2}.   For any vertex
$y \in T_n$ and $i \in S_C$, we define a probability measure
 $\mu_y^{(i)}$ on the set of spins at $y$ as follows:
\begin{equation}
\label{muy}
 \mu_y^{(i)} (\sigma_y = j) \doteq \bbP(\sigma_y =j \tc \sigma_y \leq i^*), \quad j \in S_C,
\end{equation}
with $\bbP$, as always, depending on $\l$ and a boundary condition on $T_n$ (which for clarity we have suppressed in the notation).
Note that if $x$ is a site in $T_n$ that is neighbouring to $y$, then
$\mu_y^{(i)}$ represents the marginal on $y$ of the Gibbs
measure (with
some boundary condition on the leaves of $T_n$), conditioned to
have $i$ particles at $x$.
Recall that $\|\cdot\|_{\text TV}$ denotes
the total variation distance.
\begin{lemma}
\label{vardist}
For any $k<i$, we have
$$
\|\mu_y^{(i)}-\mu_y^{(k)}\|_{\text
  TV}=\frac{\mu_y^{(0)}\left(\s_y\in[i^*+1,k^*]\right)}{\mu_y^{(0)}\left(\s_y\le
  k^*\right)}\,.
$$
\end{lemma}

\begin{proof}
By definition $\mu_y^{(i)}(\s_y=j)=\mu_y^{(0)}(\s_y=j\tc \s_y\le
i^*)$. Therefore, also recalling that $k<i$ implies $k^*>i^*$, we have
\begin{eqnarray*}
  \|\mu_y^{(i)}-\mu_y^{(k)}\|_{\text TV} & = & \frac 12 \sum_{j=0}^{i^*}\big \|\mu_y^{(i)}(\s_y=j)-\mu_y^{(k)}(\s_y=j)
  \big\| +\frac 12 \sum_{j=i^*+1}^{k^*}\mu_y^{(k)}(\s_y=j) \\
& = & \frac 12 \frac{\mu_y^{(0)}(\s_y\le k^*)-\mu_y^{(0)}(\s_y\le
  i^*)}{\mu_y^{(0)}(\s_y\le k^*)} +\frac 12 \frac{\mu_y^{(0)}(i^*+1\le
  \s_y\le k^*)}{\mu_y^{(0)}(\s_y\le k^*)}\\
& = & \frac{\mu_y^{(0)}\left(\s_y\in[i^*+1,k^*]\right)}{\mu_y^{(0)}\left(\s_y\le
  k^*\right)}.
\end{eqnarray*}
\end{proof}

Notice that if $x$ is an ancestor of $y$ then $\mu_y^{(0)}$ is nothing but the
Gibbs measure on the tree $\Tree_y^b$ rooted at $y$ with the
boundary conditions induced by those on $T_n$. If instead $y$ is
an ancestor of $x$ then $\mu_y^{(0)}$ becomes a Gibbs measure on the (non
regular) tree $T_n\setminus \Tree_x^b$. However, if $x,y$ are
sufficiently below the root of $T_n$, then $T_n\setminus
\Tree_x^b$ will coincide with a regular tree rooted at $y$ for a large
number of levels. That is all that we need to prove uniqueness below
$\left(\frac{\nep{}}{b}\right)^{\frac{1}{j_c}}$.

In what follows, given any non negative function $b \mapsto f(b)$ of the degree of
the tree $\Tree^b$, we will write $f(b)\approx 0$ if $\lim_{b\to \infty}bf(b)=0$.
\begin{lemma}
\label{vardist2}
Fix $\g'<\nep{}$ and assume $\l\le
\left(\frac{\g'}{b}\right)^{\frac{1}{j_c}}$. Then there exists $a<1$ and $n_0
\in \bbN$ such that for any
$n\ge n_0$ and any boundary condition $\t$ on the leaves of $T_n$,
\begin{equation*}
\limsup_{b\to \infty} b \mu^\t(\s_r\ge i^*+1)\le
  \begin{cases}
    0 & \text{if $i\le \inte{\frac
        C2}$}\\
  a & \text{ if $i=j_c = \ceil{\frac
        C2}$}
  \end{cases}
\end{equation*}
\end{lemma}

\begin{proof} It suffices to bound $X_n(i)$  from above for $i\le \inte{\frac
    C2}$ or $i=\ceil{\frac C2}$. In the first case, when $i\le \inte{\frac
        C2}$,  the stated bound follows easily since
(\ref{eq:R3}) and the assumed bound on $\lambda$ imply that
    for some finite constant $K$,
$$
    bX_n(i)\le \l^{i^*+1}b \leq K b^{(1 - \frac{i^*+1}{j_c})} \approx 0\,.
$$
In the second case, when $i=j_c$, set $a_\infty:=\limsup_{b\to \infty}b\hat{x}_+(b)$,
where $\hat{x}_+(b)$ is the largest fixed point of the $S$-shaped function
$x\mapsto A_\l J_2^{(\lm)}(x)$.
Due to the assumption $\l\le
\left(\frac{\g'}{b}\right)^{\frac{1}{j_c}}$, it follows that $a_\infty\le
\g'$. Because of (\ref{eq:R3}) it is enough to prove that
$a_\infty<1$. Assume the contrary. Then the fixed point equation, together
with $\l\le
\left(\frac{\g'}{b}\right)^{\frac{1}{j_c}}$,
readily implies that
$$
  a_\infty\le \g\nep{-\g\nep{-a_\infty}},
$$
which in turn implies that $a_\infty$ must be smaller than the unique
fixed point $z_0$ of the map $H$.
Since $\g\nep{-\g/\nep{}}<1$ if $\g<\nep{}$ necessarily
$z_0 <1$ and we get a contradiction.  Note that in the above proof by contradiction,
the hypothesis
  $a_\infty \ge 1$ enters as follows. If $x>1-\d$, $0<\d\ll 1$ then
  $J_2^{(\l)}(x)$ is increasing in $\l$ and so we may safely
  assume $\l= \left(\frac{\g'}{b}\right)^{\frac{1}{j_c}}$ and not just
smaller or equal.
\end{proof}

We are now ready to prove uniqueness for $\l\le
\left(\frac{\g'}{b}\right)^{\frac {1}{j_c}}$.

\begin{proof} [Proof of Theorem \ref{critical}(2)]
For simplicity we begin
with $\l=\left(\frac{\g'}{b}\right)^{\frac {1}{j_c}}$. In this case, it
follows immediately from the basic inequality (\ref{eq:R3}) that for any initial
condition, any $n\ge 2$ and any $b$ large enough, there exist constants
$c_1,\  c_2$
such that
\begin{equation}
  X_n^{\ceil{C/2}}\le c_1\nep{-c_2\,b^{\a}}\,,
\label{3.16}
\end{equation}
where $\a=1/(j_c+1)$. In another
words, recalling the probability measure $\mu_y^{(i)}$ introduced in
(\ref{muy}) and
using the obvious fact that for any $i \le C$\,,
\begin{equation*}
\mu_y^{(i)}\left([j_c+1,C]\right) \le X_n^{\ceil{C/2}}\,,
\end{equation*}
we get that the probability of having more than $j_c$ particles at
$y$ given $i$ particles at $x$ is exponentially small in $b$.

Now, recall that $T_\ell$ is the finite-tree of depth $\ell$ rooted
at $r$, and let $\t,\t'$ be two boundary conditions on the leaves of
$T_\ell$ that differ at only one vertex $v_0$. Let also $\G = \{v_0,
  v_1,\dots ,v_\ell\}$ be the unique path joining $v_0$ to the root
$r = v_\ell$. We recursively couple the corresponding measures
$\mu^\t \doteq \mu^\t_{T_{\ell}, \l}$ and $\mu^{\t'} \doteq \mu^{\t'}_{T_\ell, \l}$
by repeatedly applying the following step. Assume
that, for any pair $(\s_{v_1}, \s_{v_2})$ with $\s_{v_1}\neq \s'_{v_1}$ we can couple
$\mu^\t(\cdot\tc \s_{v_1})$ and $\mu^{\t'}(\cdot\tc \s'_{v_1})$ and call
$\nu_{\ell-1}^{\s_{v_1},\s'_{v_1}}$ the coupled measure. It is
understood that $\nu_{\ell-1}^{\s_{v_1},\s'_{v_1}}$ is concentrated
along the diagonal if $\s_{v_1}=\s'_{v_1}$.  Let
$\pi_{1}^{\t_{v_0},\t'_{v_0}}$ be the coupling of
the marginals on  of the two Gibbs measures on $v_1$ that realizes the
variation distance (\ie $\pi_{1}^{\t_{v_0},\t'_{v_0}}(\s_{v_1}\neq
\s'_{v_1})=\|\mu_{v_1}^\t -\mu_{v_1}^{\t'}\|_{TV}$).
Then we set
\begin{equation*}
  \nu_{\ell}^{\s_{v_1},\s'_{v_1}}(\s,\s') =
  \pi_{1}^{\t_{v_0},\t'_{v_0}}(\s_{v_1},\s'_{v_1})
\nu_{\ell-1}^{\s_{v_1},\s'_{v_1}}(\s_{\Tree^b_\ell\setminus v_1},\s'_{\Tree^b_\ell\setminus v_1})\,.
\end{equation*}
If we
iterate the above formula we finally get a coupling $\nu^{\t,\t'}$
such that the probability of seeing a discrepancy at the
root can be expressed as
\begin{equation}
  \sum_{\s_{v_1}\neq \s'_{v_1}\atop{\eta_{v_2}\neq \eta'_{v_2}\atop \dots }}
\pi_{1}^{\t_{v_0},\t'_{v_0}}(\s_{v_1},\s'_{v_1})\,
\pi_{2}^{\s_{v_1},\s'_{v_1}}(\eta_{v_2},\h'_{v_2})\pi_3^{\eta_{v_2},\eta'_{v_2}}\dots
\label{Markov}
\end{equation}
with self explanatory notation. If we can show that the above expression tends
to zero as $\ell\to \infty$ faster than $b^{-\ell}$  uniformly in $\t,\t'$, then
uniqueness will follow by a standard path coupling (or triangle
inequality) argument (see, for example, \cite{BubleyDyer}).

On the state space $S := [0,\dots ,C]^2$ consider a
non-homogeneous Markov chain $\{\xi_t\}_{t=0}^\ell$ with transition matrix at time $t$
given by $P_t(\xi,\xi')=\pi_t^\xi(\xi')$ and initial condition
$\xi_0=(\t_{v_0},\t'_{v_0})$.
Let also $B = \{(i, j)\in S\  : i \ge j_c+1\}\cup \{(i, j)\in S^2\ : j
\ge j_c+1\}$ be
the bad set and let $D = \{(i, i)\in S \ : i\in [0,\dots ,C]\}$ be the diagonal.
Equation (\ref{Markov})  is then nothing but the probability that the chain
does not hit $D$ within time $\ell$.

For $b$ large enough (depending only on $\g'<\nep{}$) the two key
properties of the chain, which immediately follow from Lemmas \ref{vardist} and
\ref{vardist2} and the inequality (\ref{3.16}), are the following:
\begin{align}
  \sup_t\sup_{\xi\in B^c}P_t(\xi,D^c) &\le \frac{a}{b},\quad a<1
  \label{Markov1}\\
  \sup_t\sup_{\xi}P_t(\xi,B) &\le c_1\nep{-c_2\,b^\a},\quad \a>0\,.
\label{Markov2}
\end{align}
Notice that it is not difficult to show that
$$
\sup_t\sup_{\xi\in B}P_t(\xi,D^c)\approx \l\gg 1/b\,.
$$
In other words, the probability of not entering the
diagonal $D$ in one step is suitably small (\ie smaller than $a/b$, $a < 1$)
only if we start from the good set $B^c$.
Using (\ref{Markov1}) and (\ref{Markov2}), we can immediately conclude that
\begin{eqnarray}
  \bbP(\xi_t\notin D\ \text{for all } 0 \le   t\le \ell) & \le
&
\sum_{k=0}^\ell{\ell\choose k}\left(c_1\nep{-c_2\,b^\a}\right)^k\left(\frac
ab\right)^{\ell-2k-1}\nonumber\\
& \leq & \frac ba \left(\frac ba c_1\nep{-c_2\,b^\a}+\frac ab \right)^\ell\,.
\label{Markov3}
\end{eqnarray}
The ``$-1$'' in the
exponent of $a/b$ above takes into account the fact that we may start at $x_0$ in
the bad set $B$, while the extra ``$-k$'' in the exponent accounts for the fact that for any
transition from $B$ to $B^c$ we do not necessarily have a good coupling
bound. It is clear that the right hand side of (\ref{Markov3})  tends to zero faster than
$b^{-\ell}$ as $\ell\to \infty$ because $a < 1$.
\end{proof}

\subsection{The case of $C$ even} \label{sec-ceven}
Throughout this discussion, we assume $C$ even and we set $j_c=\frac C2
+1$. Notice that $j_c=(\frac C2)^*+1$.
\begin{theorem}
\label{EVEN}
Assume $\l=\left(\g \frac{\log b}{b}\right)^{\frac
    {1}{j_c}}$ with $\g> 1/(C+2)$. Then, for any large enough $b$
  there is phase coexistence. If instead $\g< \frac {1}{C+2}$, for any large enough $b$
  there is a unique Gibbs measure.
  \end{theorem}
  \begin{proof}
Fix $\g> \frac{1}{C+2}$ and assume $\l= \left(\g \frac{\log b}{b}\right)^{\frac
    {1}{j_c}}$. We will show that the largest fixed point of $F^{(\frac C2)}_-$ is
  strictly larger than the smallest fixed point of $F^{(\frac C2)}_+$. By the usual
  argument that is enough to prove phase coexistence.

  Pick $\a$ halfway between $1/(C+2)$ and $\g$ and compute
the value $F^{(\frac C2)}_-(\frac{\a\log b}{b})$ for
  large $b$. From the
  definition we get
  \begin{equation*}
    F^{(\frac C2)}_-(\frac{\a\log b}{b})\approx
    \frac{\g\log b}{b}\,\nep{-b^{\frac{1}{C+2}-\a}}\approx
    \frac{\g\log b}{b}\gg \frac{\a\log b}{b}\,.
  \end{equation*}
Therefore there exists a fixed point of $F^{(\frac C2)}_-$ greater than
$\frac{\a\log b}{b}$. On the other hand
\begin{equation}
  F^{(\frac
  C2)}_+\left(2\g \frac{\log b}{b}\nep{-b^{1/(C+2)}}\right)\approx
\g \frac{\log b}{b}\nep{-b\l^{\frac C2}}\ll 2\g \frac{\log b}{b}\nep{-b^{1/(C+2)}}\,,
\label{fix+}
\end{equation}
so that $F^{(\frac C2)}_+$ has a fixed point smaller than
$2\g \frac{\log b}{b}\nep{-b^{1/(C+2)}}$\,, and now the first statement of the
theorem follows.

Assume now $\g< \frac {1}{C+2}$. In that case, using (\ref{eq:R3}), we
infer that, for any boundary condition and any large enough $b$,
\begin{equation*}
\mu^\t_{T_n}\left(\s_r\ge \textstyle{\frac C2 +1}\right)\le  X_n^{(C/2)}\le \nep{-b^{a}}\,,\quad a=\frac{1}{C+2}-\g
\end{equation*}
The proof of uniqueness follows now exactly the same lines of the odd
case with the difference that now the bad set is $B=\{C/2+1,\dots,C\}$
and (\ref{Markov1}), (\ref{Markov2}) are
changed into
\begin{eqnarray}
  \sup_t\sup_{\xi\in B^c}P_t(\xi,D^c) &\le& c_1\nep{-c_2\,b^\a},\quad \a>0
  \label{Markov1bis}\\
  \sup_t\sup_{\xi}P_t(\xi,B) &\le& c_1\nep{-c_2\,b^\a},\quad \a>0\,.
\label{Markov2bis}
\end{eqnarray}
\end{proof}

\subsection{First-order phase transitions for $C$ even and large $b$} \label{EVEN2}
We now turn to showing that for all even $C$ and large enough $b$ (depending on
$C$), the phase transition established in Theorem \ref{EVEN} is first-order.
At the end of Section \ref{sec-ceven} we showed that as $\l$ varies, for example, in the
    interval
$$
\left[\left(\frac{\log b}{b}\right)^{\frac
    {2}{C+2}}, \ \left( \frac{3\log b}{b}\right)^{\frac
    {2}{C+2}}\right]\,,
  $$
  the values of
  \begin{equation*}
m(\l):=\limsup_{n\to \infty}\left[\mu_{T_n}^{C}(\s_r > C/2)-\mu_{T_n}^{0}(\s_r > C/2)\right]
  \end{equation*}
  vary between
  $0$ and $\O(2\frac{\log b}{b})$\,. (Recall that the superscripts $C$ and $0$ indicate full b.c.\ and
empty b.c., respectively.)
Notice that, by monotonicity, the $\limsup_{n}$ above is attained over the
sequence of even $n$'s and that $\mu_{T_{2n}}^{C}(\s_r > C/2)$
is decreasing in $n$.

Here, we argue that in the above interval $m(\l)$ cannot be
  continuous. The starting point is the observation that, because of (\ref{fix+}), for
all
$$
\l\in \left[\left(\frac{\log b}{b}\right)^{\frac {2}{C+2}},\left(
\frac{3\log b}{b}\right)^{\frac {2}{C+2}}\right]
$$
the smallest fixed
point of $F^{(\frac C2)}_+$ is exponentially small in $b^\a$ for
some $\a>0$. Thus, in particular, there exist constants $c_1\,, c_2$ such that
\begin{equation*}
  \mu_{T_n}^{0}(\s_r > C/2)\le c_1\nep{-c_2\,b^\a}\,,\quad
  \forall n\ge 1\,.
\end{equation*}
Fix now $\d<1$ and assume that for some $n_0$,
\begin{equation*}
  \mu_{T_{2n_0}^b}^{C}(\s_r > C/2)\le \frac{\d}{b}\,.
\end{equation*}
By monotonicity that implies
\begin{equation*}
  \sup_{n\ge 2n_0}\sup_\t
  \mu_{T_n}^{C}(\s_r > C/2)\le \frac{\d}{b}\,.
\end{equation*}
Thus we can proceed with the previously described coupling argument with
(\ref{Markov1bis}) and (\ref{Markov2bis}) replaced by
\begin{eqnarray}
  \sup_{t\ge 2n_0}\sup_{\xi\in B^c}P_t(\xi,D^c) &\le& \frac{\d}{b},
  \label{Markov5bis}\\
  \sup_{t\ge 2n_0}\sup_{\xi}P_t(\xi,B) &\le& c_1\nep{-c_2\,b^\a}\,.
\label{Markov6bis}
\end{eqnarray}
and we may conclude that $m(\l)=0$.

In other words we have shown that $m(\l)>0$
implies that for all $n$,
$$
\mu_{T_{2n}}^{C}(\s_r > C/2)> \frac{\d}{b},
$$
so that
$$
m(\l)\ge  \frac{\d}{b} - c_1\nep{-c_2\,b^\a}\,.
$$
It follows now that the phase transition is first-order.

\medskip

\noindent
{\bf Acknowledgments}. The authors thank Microsoft Research, particularly the Theory Group, for its hospitality, and for facilitating this collaboration. The authors
are also thankful to Christian Borgs, Roman Koteck\'y and Ilze Ziedins for useful discussions
in the early stages of this research.

\end{document}